\documentclass[11pt]{amsart}

\issueinfo{00}{0}{Month}{Year}
\copyrightinfo{2008}{University of Calgary}
\pagespan{1}{2}
\PII{ISSN 1715-0868}

\usepackage{graphicx}
\usepackage{epstopdf}
\usepackage{amsthm,amsmath,amssymb,amsxtra,amscd}
\usepackage{verbatim}
\usepackage{indent}
\usepackage{rotating}
\usepackage{multirow}
\usepackage{multicol}
\usepackage{url}
\usepackage{algorithmic}
\usepackage{xypic}

\newtheorem{theorem}{Theorem}[section]

\newtheorem{lemma}[theorem]{Lemma}

\newtheorem{example}[theorem]{Example}
\newtheorem*{example*}{Example}

\newtheoremstyle{myexample}{3pt}{3pt}{\rmfamily}{}{\itshape}{:}{ }{\thmname{#1}\thmnumber{ #2}\thmnote{ (#3)}}
\theoremstyle{myexample}

\newtheoremstyle{myremark}{3pt}{3pt}{\rmfamily}{}{\itshape}{:}{ }{\thmname{#1}}
\theoremstyle{myremark}

\newtheorem*{observation*}{Observation}

\newtheoremstyle{conjecture}{3pt}{3pt}{\itshape}{}{\bfseries}{.}{ }{\thmname{#1}\thmnote{ (#3)}}
\theoremstyle{conjecture}

\newtheorem*{question*}{Question}

\newtheorem{theorem*}{Theorem}

\numberwithin{equation}{section}

\newcounter{algorithm}
\renewcommand{\thealgorithm}{\thesection.\arabic{algorithm}}

\renewcommand{\pmod}[1]{\ \left({\rm mod\ } #1 \right)}

\usepackage{amsthm, enumerate, float, mathrsfs, mathtools, multirow, subcaption}
\begin{document}

\title[Non-crystallographic tail-triangle C-groups of rank 4]{Non-crystallographic tail-triangle C-groups of rank 4 and interlacing number 2}

\author[M.L. Loyola]{Mark L. Loyola}
\address{Department of Mathematics, Ateneo de Manila University, Katipunan Avenue, Loyola Heights, Quezon City 1108, Philippines}
\email{mloyola@ateneo.edu}

\author[N.E.S. Leyrita]{Nonie Elvin S. Leyrita}
\address{Natural Sciences and Mathematics Division, Father Saturnino Urios University, San Francisco St, Butuan City 8600, Agusan Del Norte, Philippines}
\email{nsleyrita@urios.edu.ph}

\author[M.L.A.N. De Las Pe\~{n}as]{Ma. Louise Antonette N. De Las Pe\~{n}as}
\address{Department of Mathematics, Ateneo de Manila University, Katipunan Avenue, Loyola Heights, Quezon City 1108, Philippines}
\email{mdelaspenas@ateneo.edu}
%

\date{(date1), and in revised form (date2).}
\subjclass[2000]{20F55, 51F15, 51F25, 52B11, 52B15}	
\keywords{tail-triangle C-groups, Coxeter groups, modular reduction, semiregular polytopes, non-crystallographic}

\thanks{M.L. Loyola and M.L.A.N. De Las Pe\~{n}as would like to thank the LS Scholarly Work Faculty Grant for the research funding with Control Number SOSE 02 2019. N.E.S. Leyrita would also like to thank the Commission on Higher Education K-12 Program and Father Saturnino Urios University for the graduate scholarship grant. The authors acknowledge the anonymous reviewer for the helpful comments and valuable suggestions to improve the manuscript and its presentation.}

\begin{abstract}
    This work applies the modular reduction technique to the Coxeter group of rank 4 having a star diagram with labels 5, 3, and $k = 3, 4, 5, \text{ or } 6$. As moduli, we use the primes in the quadratic integer ring $\mathbb{Z}[\tau]$, where $\tau = \frac{1 + \sqrt{5}}{2}$, the golden ratio. We prove that each reduced group is a C-group, regardless of the prime used in the reduction. We also classify each reduced group as a reflection group over a finite field, whenever applicable.
\end{abstract}

\maketitle

\section{Introduction}
    \label{sec:intro}
    \sloppy
    A pair $(\Gamma, S)$ consisting of  a group $\Gamma$ generated by a set $S = \{\mathbf{r}_0', \mathbf{r}_1', \ldots, \mathbf{r}_{n - 1}'\}$ of $n$ involutions is called a \textit{C-group} if it satisfies the \textit{intersection condition}
    \begin{equation} 
        \langle{\mathbf{r}_i' \mid i \in I}\rangle \cap \langle{\mathbf{r}_j' \mid j \in J}\rangle = \langle{\mathbf{r}_k' \mid k \in I \cap J}\rangle
        \label{eq:intCond}
    \end{equation}
    for every pair of subsets $I$, $J$ of the indexing set $N = \{0, 1, \ldots, n - 1\}$. Here, $S$ is called the \textit{distinguished generating set} of $(\Gamma, S)$, and the cardinality $n$ of $S$ is the \textit{rank} of $(\Gamma, S)$. When $S$ is clear from context, we may refer to $\Gamma$ alone as a C-group.  As a consequence of \eqref{eq:intCond}, $S$ is necessarily a minimal generating set of $\Gamma$, and hence the involutions $\mathbf{r}_0', \mathbf{r}_1', \ldots, \mathbf{r}_{n - 1}'$ are all distinct. A Coxeter group is an example of a C-group \cite{Humphreys1990}. As a matter of fact, by the universal property of free groups, if the distinguished generators of a C-group $\Gamma$ satisfy the relations implied by a given Coxeter diagram $\mathscr{D}$, then $\Gamma$ must be a quotient of the Coxeter group determined by $\mathscr{D}$. We remark that any subgroup of $\Gamma$, or an arbitrary C-group for that matter, generated by a subset of its distinguished generating set is also a C-group. 
    
    This work focuses on groups $\Gamma = \langle{\mathbf{r}_0', \mathbf{r}_1', \mathbf{r}_2', \mathbf{r}_3'}\rangle$ of rank 4 which possess the star diagram shown in \textbf{Figure \ref{fig:53kdiagram}}, where the label $k$ of the branch connecting the nodes corresponding to $\mathbf{r}_1'$ and $\mathbf{r}_3'$ is either $3$, $4$, $5$, or $6$. That is, the order or period of the product $\mathbf{r}_i'\mathbf{r}_j'$ is precisely the label of the branch connecting the corresponding nodes. By convention, the label 3 of the branch connecting $\mathbf{r}_1'$ and $\mathbf{r}_2'$ is omitted from the diagram. In addition, the period is assumed to be 2 when there is no connecting branch between two nodes. The group $\Gamma$ is an example of a \textit{tail-triangle group} of rank 4 and \textit{interlacing number} 2, the order of $\mathbf{r}_2'\mathbf{r}_3'$. The construction of such groups using the process of amalgamation have been the focus of the works \cite{MonsonSchulte2012, MonsonSchulte2021}.
    \begin{figure}[H]
        \centering
        \includegraphics[scale = 0.41, keepaspectratio = true]{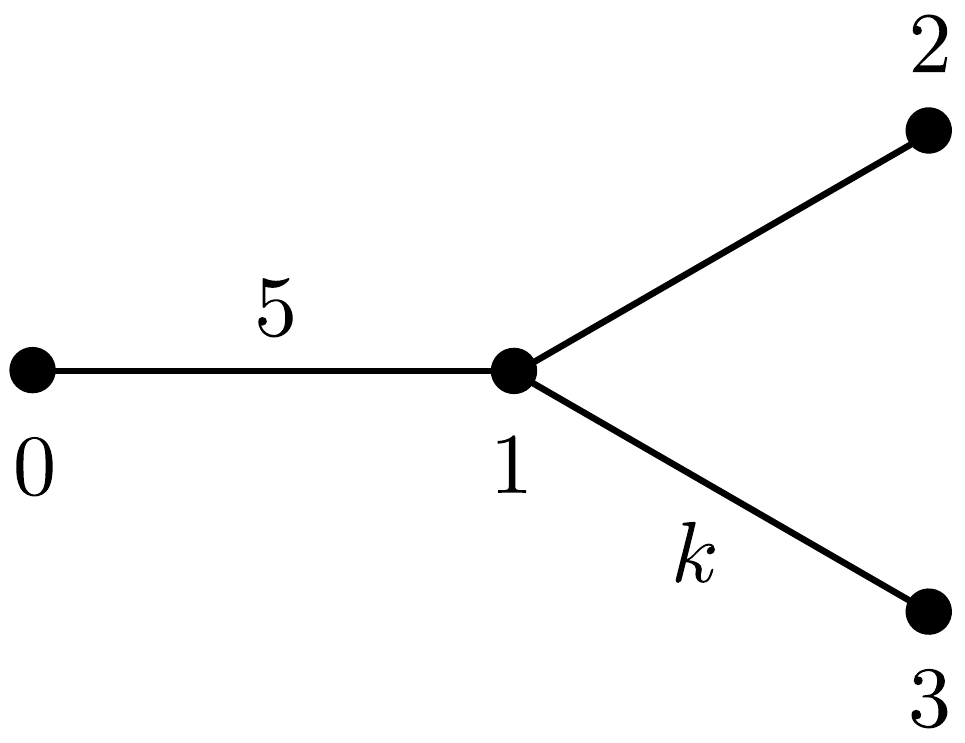} \\[2mm]
        \caption{Coxeter diagram of a star group of rank 4 and type $\{5, 3; k\}$.}
        \label{fig:53kdiagram}
    \end{figure}
    
    \noindent For brevity, we shall refer to $\Gamma$ as a \textit{star group of (rank 4 and) type $\{5, 3; k\}$} in this paper. Evidently, $\Gamma$ is a \textit{smooth quotient} of the star Coxeter group
    \begin{equation} 
        [5, 3; k] := 
        \left\langle{
            \mathbf{r}_0, \mathbf{r}_1, \mathbf{r}_2, \mathbf{r}_3 \: 
            \left| \: 
                \begin{gathered} 
                    \mathbf{r}_0^2 = \mathbf{r}_1^2 = \mathbf{r}_2^2 = \mathbf{r}_3^2 = e, \\ 
                    (\mathbf{r}_0\mathbf{r}_1)^5 = (\mathbf{r}_1\mathbf{r}_2)^3 = (\mathbf{r}_1\mathbf{r}_3)^k = e, 
                    \\ (\mathbf{r}_0\mathbf{r}_2)^2 = (\mathbf{r}_0\mathbf{r}_3)^2 = (\mathbf{r}_2\mathbf{r}_3)^2 = e
                \end{gathered}
            \right.}
        \right\rangle.
        \label{eq:53kPresentation}
    \end{equation}

    \noindent By smooth, we mean that the order of the product $\mathbf{r}_i'\mathbf{r}_j'$ in $\Gamma$ is precisely the order $m_{i, j}$ of $\mathbf{r}_i\mathbf{r}_j$ in $[5, 3; k]$.

    In general, even if $n$ is small, determining if a set of involutions satisfies \eqref{eq:intCond} is a difficult problem, especially when the generated group has no known simple structure. For one, the sheer number of subgroup intersections that needs to be checked poses a computational challenge. 
    The next lemma, which is a special case of \textbf{Lemma 4.12} in \cite{MonsonSchulte2012}, significantly reduces this number. 

    \begin{lemma}
        Let $\Gamma$ be a star group of rank 4 generated by a set $S = \{\mathbf{r}_0', \mathbf{r}_1', \mathbf{r}_2', \mathbf{r}_3'\}$ of four distinct involutions. For any $i, j \in \{0, 1, 2, 3\}$, let
        $\Gamma_i = \langle{\mathbf{r}_k' \mid k \neq i}\rangle$ and $\Gamma_{i, j} = \langle{\mathbf{r}_k' \mid k \neq i, j}\rangle$. Then $\Gamma$ is a C-group if and only if the distinguished subgroups $\Gamma_0$, $\Gamma_2$, and $\Gamma_3$ are C-groups which satisfy the intersections
        \begin{equation} 
            \Gamma_0 \cap \Gamma_2 = \Gamma_{0, 2}, \; \Gamma_0 \cap \Gamma_3 = \Gamma_{0, 3}, \; \Gamma_2 \cap \Gamma_3 = \Gamma_{2, 3}.
            \label{eq:intConds4}
        \end{equation}
        \label{lem:intConds4}
    \end{lemma}
    \noindent Note that the use of \textbf{Lemma \ref{lem:intConds4}} requires verifying beforehand that $\Gamma_0$, $\Gamma_2$, and $\Gamma_3$ are themselves C-groups. By \textbf{Proposition 2E16} of \cite{McMullenSchulte2002}, this can be accomplished by showing that the following intersections hold:
    \begin{equation} 
        \Gamma_{0, 2} \cap \Gamma_{0, 3} = \Gamma_{0, 2} \cap \Gamma_{2, 3} = \Gamma_{0, 3} \cap \Gamma_{2, 3} = \langle{\mathbf{r}_1'}\rangle.
        \label{eq:intConds3}
    \end{equation}

    C-groups are fundamental, and hence ubiquitous, in the theory of abstract and geometric polytopes. These groups, in fact, characterize the structure or substructure of certain classes of highly symmetric polytopes. For instance, the automorphism group of a regular polytope or any of its subsection is a \textit{string C-group}, that is, a C-group with a linear diagram. Conversely, given a string C-group, one may construct a unique regular polytope having the string C-group as its automorphism group. The interested reader is advised to consult the monograph \cite{McMullenSchulte2002} for further details on abstract polytopes and string C-groups. In \cite{MonsonSchulte2012}, Monson and Schulte discussed a combinatorial generalization of Wythoff construction to produce an \textit{alternating semiregular $(n + 1)$-polytope} from the cosets of a tail-triangle C-group. The automorphism group of the resulting polytope is either the tail-triangle C-group itself or properly contains this group as a normal subgroup of index 2. 

    The main objective of this current work is to construct star C-groups of type $\{5, 3; k\}$ via the method of \textit{modular reduction}. This method was discussed in great detail by Monson and Schulte in \cite{MonsonSchulte2004} and in subsequent works \cite{MonsonSchulte2007, MonsonSchulte2008}. It was mostly applied previously to various families of crystallographic string Coxeter groups. In our case, reducing the star Coxeter group $[5, 3; k]$ results to a degree 4 representation of the group over some finite field $\mathbf{F}_q$, formed by taking the quotient of the quadratic integer ring $\mathbb{Z}[\tau]$ by an ideal generated by a prime $p$ in the ring. Under this representation, the images of the distinguished generators of $[5, 3; k]$ become reflections in some orthogonal space over $\mathbf{F}_q$ for the case when $p$ is not an associate of $2$. Whenever applicable, we classify the reduced group and three of its distinguished subgroups of rank 3 as orthogonal groups. This work further extends the approach and techniques used in \cite{MonsonSchulte2009} to generate C-groups of non-crystallographic types.
\section{Preliminaries}
    \label{sec:prel}

    We begin this section with a brief discussion of the primes in the ring $\mathbf{Z}[\tau]$ followed by a review of the classification of orthogonal groups over finite fields of odd characteristic. These preliminary concepts and associated results will be referred to extensively in the succeeding sections of the paper. 

    \subsection{Primes in $\mathbf{Z}[\tau]$}
    \label{subsec:primesZtau}
    
    Let $\tau = \frac{1 + \sqrt{5}}{2}$ and consider the ring of integers
    \[ \mathbf{Z}[\tau] = \{a + b\tau \mid a, b \in \mathbf{Z}\} \]
    of the quadratic number field $\mathbf{Q}(\sqrt{5})$. For each element $a + b\tau$ of the ring, one can associate a rational integer $N(a + b\tau) = a^2 + ab - b^2$ called its \textit{norm}. A unit in $\mathbf{Z}[\tau]$ is an element of the form $\pm\tau^n$, where $n \in \mathbf{Z}$, and up to multiplication by a unit, a prime $p$ in the ring belongs to one of following three classes \cite{Dodd1983}:
    \begin{enumerate}[{Class} I.]
        \item $p = -1 + 2\tau = \sqrt{5}$ 
        \item $p$ is a rational prime such that $p \equiv \pm 2 \pmod{5}$ 
        \item $p$ is a non-rational prime such that $N(p) \equiv \pm 1 \pmod{5}$ and $|N(p)|$ is a rational prime
    \end{enumerate}
    Since the norm function $N$ remains constant on a set of associates, the primes in Class I or Class II have norms that are equal to either $0$ or $4 \!\! \mod{5}$, respectively. In addition, regardless of class, $q := |N(p)|$ is either a rational prime or the square of a rational prime. 

    A useful computational tool when dealing with the arithmetic properties of $\mathbf{Z}[\tau]$ is the \textit{Legendre symbol for $\mathbf{Z}[\tau]$}: 
    \begin{equation*}
        \left(\frac{a + b\tau}{p}\right)_{\mathbf{Z}[\tau]} =
        \begin{cases}
            \;\;\; 1 &\text{if $a + b\tau$ is a quadratic residue modulo $p$}, \\
            \;\;\; 0 &\text{if $a + b\tau \equiv 0 \pmod{p}$}, \\
            -1 &\text{if $a + b\tau$ is a non-quadratic residue modulo $p$}.
        \end{cases}
        \label{eq:LegendreSymbol}
    \end{equation*}
    This symbol generalizes the ordinary Legendre symbol $\left(\frac{a}{p}\right)_{\mathbf{Z}}$ for rational integers and satisfies the property in the next theorem, in addition to the fundamental properties it shares with $\left(\frac{a}{p}\right)_{\mathbf{Z}}$ (see \cite{Dodd1983} for a list of properties of this generalized Legendre symbol).
    \begin{theorem}
        Let $p = c + d\tau$ be an odd prime in $\mathbf{Z}[\tau]$, that is, a prime which is not an associate of $2$. Then the following holds for any rational integers $a$ and $b$:
        \[
            \left(\dfrac{a + b\tau}{p}\right)_{\mathbf{Z}[\tau]} =
            \begin{cases}
                \left(\dfrac{a}{5}\right)_{\mathbf{Z}} &\text{if $p$ is in Class I and b = 0}, \\[4mm]
                \left(\dfrac{a^2 + ab - b^2}{p}\right)_{\mathbf{Z}} &\text{if $p$ is in Class II}, \\[4mm]
                \left(\dfrac{ad^2-bcd}{|c^2 + cd - d^2|}\right)_{\mathbf{Z}} &\text{if $p$ is in Class III}.
            \end{cases}
        \]

        \label{thm:LegendreSymbol}
    \end{theorem}

    \subsection{Orthogonal groups over finite fields of odd characteristic}
    \label{subsec:orthoGroups}

    Let $n \geq 3$ and consider an $n$-dimensional vector space $V(n, q)$ over a finite field $\mathbf{F}_q$ of odd order $q$, endowed with a non-singular symmetric bilinear form $(\cdot)_V$. Thus, relative to any basis $\{v_0, v_1, \ldots, v_{n - 1}\}$ of $V(n, q)$, the Gram matrix $\mathbf{g} = [(v_i \cdot v_j)_V]$ is non-singular and symmetric. With respect to $(\cdot)_V$, the full \textit{orthogonal group}
    \[ O(n, q) = \{f \in GL(n, q) \mid (fu \cdot fv)_V = (u \cdot v)_V \text{ for all } u, v \in V(n, q)\} \]
    is defined to be the subgroup of $GL(n, q)$ consisting of all invertible endomorphisms of $V(n, q)$ that preserve $(\cdot)_V$.

    We recall from \cite{Wilson2009} that $(\cdot)_V$ belongs to one of two equivalence classes of non-singular symmetric bilinear forms on $V(n, q)$ under the action of $GL(n, q)$. When $n$ is even, these two non-equivalent classifications give rise to the non-isomorphic orthogonal groups denoted by $O(n, q, \varepsilon)$, where $\varepsilon = 1$ or $-1$. In particular, when $n$ is even, the parameter $\varepsilon$ specifies whether $(-1)^{\frac{n}{2}}\det{\mathbf{g}}$ is a square or a non-square, respectively, in $\mathbf{F}_q^*$, the multiplicative group of units of $\mathbf{F}_q$. On the other hand, when $n$ is odd, we obtain only the orthogonal group $O(n, q, 0)$. To summarize, $O(n, q)$ can be classified as one of the following types of orthogonal groups over finite fields:
    \begin{equation}
        O(n, q, \varepsilon) =
        \begin{cases}
            O(n, q, 1) &\text{if $n$ is even and $(-1)^{\frac{n}{2}}\det{\mathbf{g}} \in (\mathbf{F}_q^*)^2$}, \\
            O(n, q, 0) &\text{if $n$ is odd}, \\
            O(n, q, -1) &\text{if $n$ is even and $(-1)^{\frac{n}{2}}\det{\mathbf{g}} \notin (\mathbf{F}_q^*)^2$}. \\
        \end{cases}
        \label{eq:Onq}
    \end{equation}

    Given a non-zero anisotropic vector $v$ in $V(n, q)$, that is, a vector $v$ such that $(v \cdot v)_V \neq 0$, we define the involutory map
    \[ r_v(x) = x - 2\dfrac{(x \cdot v)_V}{(v \cdot v)_V}v \]
    in $O(n, q, \varepsilon)$ and call it the \textit{reflection with root $v$}. Note that for any non-zero scalar $\alpha \in \mathbf{F}_q$, we have $r_v = r_{\alpha v}$. The reflection $r_v$ negates the space spanned by $v$ and fixes pointwise the orthogonal complement $v^{\perp}$. Clearly, $\det{r_v} = -1$. Reflections play an integral role in the theory of orthogonal groups since an arbitrary element in $O(n, q, \varepsilon)$ can be written as a product of reflections \cite{Grove2002}. 
    
    An essential map in the theory of orthogonal groups is the \textit{spinor norm} $\theta : O(n, q, \varepsilon) \to \mathbf{F}_q^*/(\mathbf{F}_q^*)^2$ which sends an element $r = r_{v_{k_1}}  r_{v_{k_2}} \cdots r_{v_{k_m}}$ written as a product of reflections in $O(n, q, \varepsilon)$ to the coset of $\mathbf{F}_q^*$ containing the product $(v_{k_1} \cdot v_{k_1})_V(v_{k_2} \cdot v_{k_2})_V \cdots (v_{k_m} \cdot v_{k_m})_V$. This map is a surjective homomorphism and does not depend on the reflections used in the factorization \cite{Grove2002}. Consequently, if $r_v$ is a reflection in the orthogonal group, then it must belong to exactly one of the subgroups 
    \begin{equation} 
        \begin{gathered}
        O_1(n, q, \varepsilon) = \langle{r_v \mid (v \cdot v)_V \in (\mathbf{F}_q^*)^2}\rangle \\
        \text{ and } \\
        O_2(n, q, \varepsilon) = \langle{r_v \mid (v \cdot v)_V \notin (\mathbf{F}_q^*)^2}\rangle, 
        \end{gathered}
        \label{eq:O1nqO2nq}
    \end{equation}
    which we shall encounter again in the next section. We note that when $n$ is odd, these two subgroups are non-isomorphic. When $n$ is even, on the other hand, these two subgroups are isomorphic and, in fact, conjugate in $GL(n, q)$ \cite{MonsonSchulte2004}.

    For $n = 3$ or $4$, Table \ref{tbl:relevantGroups} summarizes the orders of the groups described above. Observe, in particular, that the special subgroup $O_1(n, q, \varepsilon)$ or $O_2(n, q, \varepsilon)$ is an index $2$ subgroup of the full orthogonal group $O(n, q, \varepsilon)$. 

    In this work, we are mainly concerned with a subgroup $G$ of $O(4, q, \varepsilon)$ generated by a set of four distinct reflections $r_0$, $r_1$, $r_2$, $r_3$ with roots $v_0$, $v_1$, $v_2$, $v_3$, respectively, that satisfy the relations implicit in the Coxeter diagram of $[5, 3; k]$ with $k = 3, 4, 5, \text{ or } 6$. That is, the pairwise products of distinct reflections $r_0r_1$, $r_1r_2$, $r_1r_3$ have orders $5$, $3$, $k$, respectively, while the rest have order $2$. We impose the condition that $G$ must act irreducibly on the ambient space $V(4, q)$. Given that the diagram of the generating reflections is connected, this imposition is equivalent to condition that the Cartan matrix $\mathbf{c} = \left[2\dfrac{(v_j \cdot v_i)_V}{(v_i \cdot v_i)_V}\right]$ is non-singular \cite{MonsonSchulte2004}. Thus, we may safely assume that the roots $v_0$, $v_1$, $v_2$, $v_3$ form a basis for $V(4, q)$.

    \begin{table}[H]
        \renewcommand*{\arraystretch}{1.3}
        \centering
        \begin{tabular}{|c|c|}
            \hline
            Group & Order \\
            \hline
            \hline
            $\mathcal{A}_3^p$ & $24$ \\
            \hline
            $\mathcal{B}_3^p$ & $48$ \\
            \hline
            $\mathcal{H}_3^p$ & $120$ \\
            \hline
            $[3, 6]_{(s, 0)}, \; s \geq 2$ & $12s^2$ \\
            \hline
            $O(3, q, 0)$ & $2q(q^2 - 1)$ \\
            \hline
            $O_1(3, q, 0)$ & \multirow{2}{*}{$q(q^2 - 1)$} \\
            \cline{1-1}
            $O_2(3, q, 0)$ & \\
            \hline      
            $O(4, q, 1)$ & $2q^2(q^2 - 1)^2$ \\
            \hline
            $O_1(4, q, 1)$ & \multirow{2}{*}{$q^2(q^2 - 1)^2$} \\
            \cline{1-1}
            $O_2(4, q, 1)$ & \\
            \hline  
            $O(4, q, -1)$ & $2q^2(q^2 + 1)(q^2 - 1)$ \\
            \hline
            $O_1(4, q, -1)$ & \multirow{2}{*}{$q^2(q^2 + 1)(q^2 - 1)$} \\
            \cline{1-1}
            $O_2(4, q, -1)$ & \\
            \hline                               
        \end{tabular}
        \caption{Relevant groups in the classification of the reduction modulo an odd prime $p$ of the star Coxeter group $[5, 3; k]$ and its distinguished subgroups of rank 3.}
        \label{tbl:relevantGroups}
    \end{table}

\section{Modular Reduction of $[5, 3; k]$ with $k = 3, 4, 5, 6$}
\label{sec:modred}

Given a basis $\{a_0, a_1, a_2, a_3\}$ for the real vector space $\mathbf{R}^4$, define the symmetric bilinear form
\[ (a_i \cdot a_j) = -\cos\frac{\pi}{m_{i, j}}, \]
where $m_{i, j}$ denotes the exponent of the product $\mathbf{r}_i\mathbf{r}_j$ in the group presentation \eqref{eq:53kPresentation}. Consider the rescaled basis vectors $v_i = c_ia_i$ for $i = 0, 1, 2, 3$, where
\begin{equation}
    c_0 = 2, \:
    c_1 = 2\tau, \:
    c_2 = 2\tau, \:
    c_3 = 
    \begin{cases} 
        2\tau & \text{if $k = 3$}, \\ 
        2\sqrt{2}\tau & \text{if $k = 4$}, \\ 
        2\tau^2 & \text{if $k = 5$}, \\ 
        2\sqrt{3}\tau & \text{if $k = 6$},
    \end{cases} 
    \label{eq:scale}
\end{equation}
and define the involutory endomorphisms $r_i(x) = x - 2\dfrac{(x \cdot v_i)}{(v_i \cdot v_i)}v_i$ of $\mathbf{R}^4$. The matrices which represent these involutions with respect to the rescaled basis $B = \{v_0, v_1, v_2, v_3\}$ are
\begin{equation}
    \begin{gathered}
        \mathbf{r}_0 =
        \begin{bmatrix}
            -1 & \tau^2 & 0 & 0 \\
            0 & 1 & 0 & 0 \\
            0 & 0 & 1 & 0 \\
            0 & 0 & 0 & 1
        \end{bmatrix}\!, \:      
        \mathbf{r}_2 =
        \begin{bmatrix}
            1 & 0 &  0 & 0 \\
            0 & 1 &  0 & 0 \\
            0 & 1 & -1 & 0 \\
            0 & 0 &  0 & 1
        \end{bmatrix}\!, \:
        \:\: \mathbf{r}_3 =
        \begin{bmatrix}
            1 & 0 & 0 &  0 \\
            0 & 1 & 0 &  0 \\
            0 & 0 & 1 &  0 \\
            0 & 1 & 0 & -1
        \end{bmatrix}\!, \\
        \mathbf{r}_1 =
        \begin{bmatrix}
            1 &  0 & 0 & 0 \\
            1 & -1 & 1 & \rho_k \\
            0 &  0 & 1 & 0 \\
            0 &  0 & 0 & 1
        \end{bmatrix}\!, \:
        \text{where }
        \rho_k = 
        \begin{cases} 
            1 & \text{if $k = 3$}, \\ 
            2 & \text{if $k = 4$}, \\ 
            \tau^2 & \text{if $k = 5$}, \\ 
            3 & \text{if $k = 6$}.
        \end{cases}
    \end{gathered}
    \label{eq:refMats}
\end{equation}
These matrices, whose entries clearly lie in $\mathbf{Z}[\tau]$, generate a subgroup $G$ of $GL(4, \mathbf{R})$ that is isomorphic to the star Coxeter group $[5, 3; k]$ (see \S{5.3} - \S{5.4} of \cite{Humphreys1990}).

Reducing the entries of the matrices in \eqref{eq:refMats} modulo a prime $p$ in $\mathbf{Z}[\tau]$ yields the group 
\[ G^p = \langle{\mathbf{r}_0, \mathbf{r}_1, \mathbf{r}_2, \mathbf{r}_3}\rangle^p := \langle{\mathbf{r}_i \!\!\!\! \mod{p} \mid i = 0, 1, 2, 3}\rangle  \] 
which acts on the 4-dimensional space spanned by $B$ over the finite field $\mathbf{F}_q = \mathbf{Z}[\tau]/(p)$ of order $q = |N(p)|$. In this context, $G^p$ is called a \textit{modular reduction} of $G$. In the ensuing discussions, we consider the case when $p$ is an \textit{even prime} (an associate of $2$) separately from the case when $p$ is an \textit{odd prime} (not an associate of $2$).

\subsection{$p$ is an even prime}
\label{subsec:evenPrime}

We may assume, without loss of generality, that $p = 2$. In this case, since $q = |N(2)| = 4$, all entries of $\mathbf{r}_i \!\! \mod{p}$ lie in the finite field $\mathbf{F}_4$. Calculations using GAP \cite{GAP2020} reveal that, for $k = 3, 4, 5, 6$, the reduced group $G^2$ is isomorphic to the semidirect product $C_2^4 \rtimes A_5$ of order 960. In addition, $G^2$ is a smooth quotient of $[5, 3; k]$, except when $k = 6$. In this exceptional case, the order of the product $\mathbf{r}_1\mathbf{r}_3 \!\! \mod{2}$ is $3$, and not $6$. 

The isomorphism types of the distinguished subgroups $G_0^2$, $G_2^2$, and $G_3^2$ are listed in \textbf{Table \ref{tbl:GpG023pClassificationEven}}. Notice that the entries in the table for $k = 3$ and $k = 6$ are exactly the same. This is expected since for both values of $k$, we have $\rho_k \equiv 1 \pmod{2}$. Finally, it can be easily verified using GAP that these distinguished subgroups satisfy the intersection condition \eqref{eq:intConds4} for each of the various cases of $k$.

We summarize the above results in the following theorem:
\begin{theorem}
    The group $G^2$, generated by the reductions modulo $2$ of the matrices $\mathbf{r}_0, \mathbf{r}_1, \mathbf{r}_2, \mathbf{r}_3$ in \eqref{eq:refMats}, is a C-group that is isomorphic to the semidirect product $C_2^4 \rtimes A_5$.

    \label{thm:GpTailTriangleCGroupEven}
    \begin{table}[H]
        \renewcommand*{\arraystretch}{1.6}
        \centering
        \begin{tabular}{|c|c|c|c|c|}
            \hline
            $k$ & $G^2$ & $G_0^2$ & $G_2^2$ & $G_3^2$ \\
            \hline
            \hline
            $3$ & \multirow{4}{*}{$C_2^4 \rtimes A_5$} & $S_4$ & $A_5$ & \multirow{4}{*}{$A_5$} \\
            \cline{1-1}
            \cline{3-4}
            $4$ &  & $S_4$ & \multirow{2}{*}{$(C_2^4 \rtimes C_5) \rtimes C_2$} & \\
            \cline{1-1}
            \cline{3-3}
            $5$ & & $A_5$ & & \\
            \cline{1-1}
            \cline{3-4}
            $6$ & & $S_4$ & $A_5$ & \\
            \hline
        \end{tabular}
        \caption{The reduction of $G \simeq [5, 3; k]$ modulo $2$.}
        \label{tbl:GpG023pClassificationEven}
    \end{table}
\end{theorem}


\subsection{$p$ is an odd prime}
\label{subsec:oddPrime}

Let $p = c + d\tau$ be an odd prime in $\mathbb{Z}[\tau]$. If the form $(\cdot)_p := (\cdot) \! \mod{p}$ is non-singular, which ultimately depends on the prime $p$ used in the reduction, then $\mathbf{r}_i \!\! \mod{p}$ represents a reflection with root $v_i$. Consequently, $G^p$ becomes a reflection group whose Gram and Cartan matrices are the reductions modulo $p$ of
\begin{equation}
    \begin{gathered}
        \mathbf{g} =
        \begin{bmatrix}
            4 & -2\tau^2 & 0 & 0 \\
            -2\tau^2 & 4\tau^2 & -2\tau^2 & -2\tau^2\rho_k \\
            0 & -2\tau^2 & 4\tau^2 & 0 \\
            0 & -2\tau^2\rho_k & 0 & 4\tau^2 \rho_k
        \end{bmatrix}\!, \:
        \mathbf{c} =
        \begin{bmatrix}
            2 & -\tau^2 & 0 & 0 \\
            -1 & 2 & -1 & -\rho_k \\
            0 & -1 & 2 & 0 \\
            0 & -1 & 0 & 2
        \end{bmatrix}
    \end{gathered}
    \label{eq:GramCartan}
\end{equation}
with determinants
\begin{equation}
    \begin{gathered}
        \det{\mathbf{g}} = 2^6\tau^4\rho_k(1 - \tau^2 \rho_k), \;  
        \det{\mathbf{c}} = 2^2\tau^{-2}(1 - \tau^2\rho_k),
    \end{gathered}
    \label{eq:GramCartanDets}
\end{equation}
respectively. Moreover, a series of straightforward calculations shows that $G^p$ is a smooth quotient of $[5, 3; k]$ for $k = 3, 4, 5, 6$.

The following theorem states that $G^p$ is either the full orthogonal group $O(4,q, \varepsilon)$ or the special reflection group $O_1(4, q, \varepsilon)$ described in the previous section. 

\begin{theorem}
    Let $p$ be an odd prime in $\mathbf{Z}[\tau]$, and suppose that the reductions modulo $p$ of the matrices $\mathbf{g}$ and $\mathbf{c}$ in \eqref{eq:GramCartan} are non-singular. Then the group $G^p$, generated by the reductions modulo $p$ of the matrices $\mathbf{r}_0, \mathbf{r}_1, \mathbf{r}_2, \mathbf{r}_3$ in \eqref{eq:refMats}, is an orthogonal group of type either $O(4, q, \varepsilon)$ or $O_1(4, q, \varepsilon)$, where $\varepsilon = \left(\frac{\det{\mathbf{g}}}{p}\right)_{\mathbf{Z}[\tau]}$.
    \label{thm:groupTypeGp}
\end{theorem}

\begin{proof}
    To prove the theorem, we employ \textbf{Theorem 3.1} of \cite{MonsonSchulte2004}, which here implies that $G^p$ must be one of the following irreducible reflection groups:
    \begin{itemize}
        \item the orthogonal group $O(4, q', \varepsilon)$, $O_1(4, q', \varepsilon)$, or $O_2(4, q', \varepsilon)$, where $q' > 1$ is a divisor of $q$; or
        \item the reduction modulo $p$ of the standard representation of the finite irreducible Coxeter group $\mathcal{A}_4$, $\mathcal{B}_4$, $\mathcal{D}_4$, $\mathcal{F}_4$, or $\mathcal{H}_4$.
    \end{itemize}
    
    From \textbf{Table 1} of \cite{MonsonSchulte2004}, we see that the order of the product of two reflections in any of the reduced groups $\mathcal{A}_4^p$, $\mathcal{B}_4^p$, $\mathcal{D}_4^p$, $\mathcal{F}_4^p$ is at most 4. Since the product $\mathbf{r}_0\mathbf{r}_1 \!\! \mod{p}$ in $G^p$ has order 5, we immediately rule out these four crystallographic Coxeter groups. We also rule out $\mathcal{H}_4^p$ since it cannot be generated by four reflections with a star diagram of type $\{5, 3; k\}$, as can be routinely checked using GAP. Consequently, $G^p$ must have an orthogonal group type. 

    The unit entry $\tau^2$ in $\mathbf{r}_0$ implies that $q' = q$. That is, $G^p$ must be an orthogonal group over $\mathbf{F}_q$ itself, and not over a proper subfield. In addition, since the root $v_0$ satisfies $(v_0 \cdot v_0)_p = 4 \!\! \mod{p}$, which is clearly a square in $\mathbf{F}_q^*$ for any $p$, then $G^p$ must be either  $O(4, q, \varepsilon)$ or $O_1(4, q, \varepsilon)$.
\end{proof}

From the Gram matrix $\mathbf{g}$ in \eqref{eq:GramCartan}, we see that $(v_0 \cdot v_0)$, $(v_1 \cdot v_1)$, and $(v_2 \cdot v_2)$ are all squares in $\mathbf{Z}[\tau]$. Thus, determining whether $G^p$ has type $O_1(4, q, \varepsilon)$ or $O(4, q, \varepsilon)$, respectively, rests on whether the reduction modulo $p$ of $(v_3 \cdot v_3) = 4\tau^2\rho_k$
is a square or not in $\mathbf{F}_q^*$. This is, of course, equivalent to determining whether $\delta := \left(\frac{\rho_k}{p}\right)_{\mathbf{Z}[\tau]}$ evaluates to $1$ or $-1$, respectively. In particular, if $k = 3$ or $5$, then $\delta = 1$, regardless of $p$. 

\begin{example} 
    Let $k = 6$. Then $G = \langle{\mathbf{r}_i \mid i = 0, 1, 2, 3}\rangle$ is isomorphic to the star Coxeter group $[5, 3; 6]$. From \eqref{eq:GramCartanDets}, we have $\det{\mathbf{g}} = 2^6\tau^4(-3\tau^4)$ and $\det{\mathbf{c}} = 2^2\tau^{-2}(-\tau^4)$. For an odd prime $p$ in $\mathbf{Z}[\tau]$, we use \textbf{Theorem \ref{thm:groupTypeGp}}, whenever applicable, to classify the type of the reduced group $G^p$ according to the class where $p$ belongs. 
    
    The equalities
    \begin{equation*}
        \begin{gathered} 
            \varepsilon = \left(\frac{\det{\mathbf{g}}}{p}\right)_{\mathbf{Z}[\tau]} = \left(\frac{-3}{p}\right)_{\mathbf{Z}[\tau]}, \;
            \left(\frac{\det{\mathbf{c}}}{p}\right)_{\mathbf{Z}[\tau]} = \left(\frac{-1}{p}\right)_{\mathbf{Z}[\tau]}, \text{ and } 
            \delta = \left(\frac{3}{p}\right)_{\mathbf{Z}[\tau]}, 
        \end{gathered}
        \label{eq:LegendreSymbols536}
    \end{equation*}
    which follow from the properties of the Legendre symbol for $\mathbf{Z}[\tau]$ will be useful. It follows that the only time \textbf{Theorem \ref{thm:groupTypeGp}} cannot be applied is when $p$ is a prime in Class II that is an associate of $3$. This special case will have to be handled separately. 

    \begin{enumerate}[{Class} I.]
        \item Up to associates, $p = -1 + 2\tau = \sqrt{5}$. By \textbf{Theorem \ref{thm:LegendreSymbol}}, we have
            \[ 
                \varepsilon = \left(\dfrac{-3}{5}\right)_{\mathbf{Z}} = -1 
                \text{ and }
                \delta = \left(\frac{3}{5}\right)_{\mathbf{Z}} = -1, 
            \]
            which gives us $G^p \simeq O(4, 5, -1)$.
            
        \item Up to associates, $p$ is a rational prime such that $p \equiv \pm{2} \pmod{5}$. We consider two subcases:
        \begin{enumerate}
            \item If $p$ is an associate of 3, then a computation in GAP shows that $G^p \simeq C_2 \times [C_3^6 \rtimes (C_2 \times A_5)]$ is a group of order 174,960 with distinguished subgroups $G_0^p \simeq (C_3 \times C_3) \rtimes D_6$, $G_2^p \simeq C_3^4 \rtimes D_{10}$, and $G_3^p \simeq  C_2 \times A_5$.
                        
            \item If $p$ is not an associate of 3, on the other hand, then by \textbf{Theorem \ref{thm:LegendreSymbol}}, we have
            \[ 
                \varepsilon = \left(\dfrac{9}{p}\right)_{\mathbf{Z}} = 1 
                \text{ and }
                \delta = \left(\frac{9}{p}\right)_{\mathbf{Z}} = 1, 
            \]        
            which gives us $G^p \simeq O_1(4, p^2, 1)$.
        \end{enumerate}

        \item In this class, $p$ is a non-rational prime with $q \equiv \pm{1} \pmod{5}$. By \textbf{Theorem \ref{thm:LegendreSymbol}}, we have
            \[ 
                \varepsilon = \left(\dfrac{-3}{q}\right)_{\mathbf{Z}} = \left(\dfrac{-1}{q}\right)_{\mathbf{Z}}\left(\dfrac{3}{q}\right)_{\mathbf{Z}}
                \text{ and }
                \delta = \left(\frac{3}{q}\right)_{\mathbf{Z}},
            \]
            which can be easily calculated using the relation
            \[
                \left(\frac{3}{q}\right)_{\mathbf{Z}} = \left(\frac{-1}{q}\right)_{\mathbf{Z}}\left(\frac{q}{3}\right)_{\mathbf{Z}},
            \]
            where 
            \[
                \left(\frac{-1}{q}\right)_{\mathbf{Z}} = 
                \begin{cases}
                    \;\;\: 1 &\text{if $q \equiv 1 \pmod{4}$}, \\
                    -1 &\text{if $q \equiv -1 \pmod{4}$},
                \end{cases}
            \]
            and 
            \[
                \left(\frac{q}{3}\right)_{\mathbf{Z}} = 
                \begin{cases}
                    \;\;\: 1 &\text{if $q \equiv 1 \pmod{3}$}, \\
                    -1 &\text{if $q \equiv -1 \pmod{3}$},
                \end{cases}
            \] 
            implied by the Law of Quadratic Reciprocity for $\mathbf{Z}$.          Combining these results with the assumption that $q \equiv \pm{1} \pmod{5}$ and using the Chinese Remainder Theorem yield
            \[
                G^p \simeq 
                \begin{cases}
                    O(4, q, 1) &\text{if $q \equiv 19, 31 \pmod{60}$}, \\
                    O(4, q, -1) &\text{if $q \equiv 29, 41 \pmod{60}$}, \\                    
                    O_1(4, q, 1) &\text{if $q \equiv 1, 49 \pmod{60}$}, \\
                    O_1(4, q, -1) &\text{if $q \equiv 11, 59 \pmod{60}$}. \\
                \end{cases}
            \] 
    \end{enumerate}
    \label{ex:536GroupType}
\end{example}

Performing a series of computations for $k = 3, 4, 5$ that is similar to what we have accomplished for $k = 6$ above yields the classification of $G^p$ in \textbf{Table \ref{tbl:GpG023pClassificationOdd}}. Observe that the only other case in which $G^p$ is not an orthogonal group is when $k = 5$ and $p$ is a prime in Class I. In this case, one can easily verify that the corresponding Gram matrix is singular. 

It is important to state that the classification of $G^p$ depends not only on the prime $p$ used in the reduction, but also on the choice of scaling factors considered from the very beginning to rescale the basis $\{a_0, a_1, a_2, a_3\}$ to obtain $\{v_1, v_2, v_3, v_4\}$. For instance, in the case $k = 3$, if we multiply each $c_i$ in \eqref{eq:scale} by a factor of $\sqrt{2}$, we get $G^{\sqrt{5}} \simeq O_2(4, 5, -1)$ in place of $O_1(4, 5, -1)$, and $G^p \simeq O\left(4, q, \left(\frac{cd}{q}\right)_{\mathbf{Z}}\right)$ in place of $O_1\left(4, q, \left(\frac{cd}{q}\right)_{\mathbf{Z}}\right)$ whenever $p = c + d\tau$ is a prime in Class III with $q \equiv 11, 19, 21, 29  \pmod{40}$. 

\begin{table}[H]
    \renewcommand*{\arraystretch}{1.6}
    \scriptsize
    \centering
    \begin{tabular}{|c|c|l|c|c|c|c|}
        \hline
        $k$ & \multicolumn{2}{c|}{Restrictions on $p$} & $G^p$ & $G_0^p$ & $G_2^p$ & $G_3^p$ \\
        \hline
        \hline
        \multirow{4}{*}{$3$} & I & none & $O_1(4, 5, -1)$ & \multirow{4}{*}{$\mathcal{A}_3^p$} & \multirow{4}{*}{$\mathcal{H}_3^p$} & \multirow{20}{*}{$\mathcal{H}_3^p$} \\
        \cline{2-4}
        & \multirow{2}{*}{II} & $p \equiv 13, 17 \pmod{20}$ & $O_1(4, p^2, 1)$ & & & \\
        \cline{3-4}
        & & $p \equiv 3, 7 \pmod{20}$ & $O_1(4, p^2, -1)$ & & & \\
        \cline{2-4}
        & III & none & $O_1\left(4, q, \left(\frac{cd}{q}\right)_{\mathbf{Z}}\right)$ & & & \\
        \cline{1-6}  
        \multirow{5}{*}{$4$} & I & none & $O(4, 5, 1)$ & \multirow{5}{*}{$\mathcal{B}_3^p$} & $O(3, 5, 0)$ & \\
        \cline{2-4}
        \cline{6-6}
        & \multirow{2}{*}{II} & $p \equiv 13, 17 \pmod{20}$ & $O_1(4, p^2, 1)$ & & \multirow{2}{*}{$O_1(3, p^2, 0)$} & \\
        \cline{3-4}
        & & $p \equiv 3, 7 \pmod{20}$ & $O_1(4, p^2, -1)$ & & & \\
        \cline{2-4}
        \cline{6-6}
        & \multirow{2}{*}{III} & $q \equiv 11, 19, 21, 29 \pmod{40}$ & $O\left(4, q, \left(\frac{2cd}{q}\right)_{\mathbf{Z}}\right)$ & & $O(3, q, 0)$ & \\
        \cline{3-4}
        \cline{6-6}
        & & $q \equiv 1, 9, 31, 39 \pmod{40}$ & $O_1\left(4, q, \left(\frac{2cd}{q}\right)_{\mathbf{Z}}\right)$ & & $O_1(3, q, 0)$ & \\
        \cline{1-6}
        \multirow{4}{*}{$5$} & I & none & $C_5^3 \rtimes (C_2 \times A_5)$ & \multirow{4}{*}{$\mathcal{H}_3^p$} & $O_1(3, 5, 0)$ & \\
        \cline{2-4}
        \cline{6-6}
        & \multirow{2}{*}{II} & $p \equiv 3, 7 \pmod{20}$ & $O_1(4, p^2, 1)$ & & \multirow{2}{*}{$O_1(3, p^2, 0)$} & \\
        \cline{3-4}
        & & $p \equiv 13, 17 \pmod{20}$ & $O_1(4, p^2, -1)$ & & & \\
        \cline{2-4}
        \cline{6-6}
        & III & none & $O_1\left(4, q, \left(\frac{2cd + d^2}{q}\right)_{\mathbf{Z}}\right)$ & & $O_1(3, q, 0)$ & \\
        \cline{1-6}
        \multirow{7}{*}{$6$} & I & none & $O(4, 5, -1)$ & $[3, 6]_{(5, 0)}$ & $O(3, 5, 0)$ & \\
        \cline{2-6}
        & \multirow{2}{*}{II} & $p = 3$ & $C_2 \times [C_3^6 \rtimes (C_2 \times A_5)]$ & \multirow{2}{*}{$[3, 6]_{(p, 0)}$} & $C_3^4 \rtimes D_{10}$ & \\
        \cline{3-4}
        \cline{6-6}
        & & $p \neq 3$ & $O_1(4, p^2, 1)$ & & $O_1(3, p^2, 0)$ & \\
        \cline{2-6}
        & \multirow{4}{*}{III} & $q \equiv 19, 31 \pmod{60}$ & $O(4, q, 1)$ & \multirow{4}{*}{$[3, 6]_{(q, 0)}$} & \multirow{2}{*}{$O(3, q, 0)$} & \\
        \cline{3-4}
        & & $q \equiv 29, 41 \pmod{60}$ & $O(4, q, -1)$ & & & \\
        \cline{3-4}
        \cline{6-6}
        & & $q \equiv 1, 49 \pmod{60}$ & $O_1(4, q, 1)$ & & \multirow{2}{*}{$O_1(3, q, 0)$} & \\
        \cline{3-4}
        & & $q \equiv 11, 59 \pmod{60}$ & $O_1(4, q, -1)$ & & & \\                
        \hline    
    \end{tabular}
    \caption{The reduction of $G \simeq [5, 3; k]$ modulo $p = c + d\tau$, an odd prime in $\mathbf{Z}[\tau]$.}
    \label{tbl:GpG023pClassificationOdd}
\end{table}

The next theorem classifies the types of the distinguished subgroups $G_0^p$, $G_2^p$, and $G_3^p$ of rank 3. 

\begin{theorem}
    Let $p$ be an odd prime in $\mathbf{Z}[\tau]$, and  let $G_i^p$, where $i = 0, 2, \text{or } 3$, be the distinguished subgroup of $G^p$ that is generated by the reductions modulo $p$ of the matrices $\mathbf{r}_j$ with $j \neq i$ in \eqref{eq:refMats}. Then the following statements hold:
    \begin{enumerate}
        \item $G_0^p$ is the reduction modulo $p$ of the standard representation of the finite irreducible Coxeter group $\mathcal{A}_3$, $\mathcal{B}_3$, or $\mathcal{H}_3$, if $k = 3, 4, \text{ or } 5$, respectively; or is the automorphism group $[3, 6]_{(s, 0)}$ of the regular torus $\{3, 6\}_{(s, 0)}$, if $k = 6$, where $s = 5, p, q$ according as $p$ is in Class I, II, III, respectively.
        \item $G_2^p$ is an orthogonal group of type $O(3, q, 0)$ or $O_1(3, q, 0)$, except when $k = 6$ and $p$ is an associate of $3$.
        \item $G_3^p$ is the reduction modulo $p$ of the standard representation of the finite irreducible Coxeter group $\mathcal{H}_3$.
    \end{enumerate}
    \label{thm:groupTypeGp023}
\end{theorem}

\begin{proof}
    Except for the case $i = 0$ and $k = 6$, or the case $i = 2$, $k = 6$, and $p$ is an associate of $3$, the reductions modulo $p$ of the submatrices obtained by deleting the $(i + 1)$st row and the $(i + 1)$st column of $\mathbf{g}$ and $\mathbf{c}$ in \eqref{eq:GramCartan} are non-singular. Consequently, except for these aforementioned cases, \textbf{Theorem 3.1} of \cite{MonsonSchulte2004}, where $n = 3$, may be applied. The smoothness of $G^p$ as a quotient of the group $[5, 3; k]$ is enough to conclude that $G_0^p$ or $G_3^p$ corresponds to one of $\mathcal{A}_3^p$, $\mathcal{B}_3^p$, and $\mathcal{H}_3^p$. On the other hand, using a series of arguments similar to what we have used in the proof of \textbf{Theorem \ref{thm:groupTypeGp}} above, we obtain the desired conclusion for $G_2^p$. 

    Thus, only the distinguished subgroup $G_0^p$ when $k = 6$ remains to be classified. First, note that $G_0^p$ is a finite smooth quotient of $[3, 6]$, the automorphism group of the tessellation $\{3, 6\}$ of the Euclidean plane by regular triangles. It follows from the list of finite quotients of the Coxeter group $[3, 6]$ (see \textbf{Table I} of \cite{Senechal1985}) that $G_0^p$ is isomorphic to either $[3, 6]_{(s, 0)} \simeq (C_s \times C_s) \rtimes D_6$, the automorphism group of the regular torus $\{3, 6\}_{(s, 0)}$, or to $[3, 6]_{(s, s)} \simeq (C_{3s} \times C_s) \rtimes D_6$, the automorphism group of the regular torus $\{3, 6\}_{(s, s)}$ (see \cite{McMullenSchulte2002} for further details on these regular toroidal polyhedra). In either group decomposition, $\mathbf{x} \!\! \mod{p}$ and $\mathbf{x}^{-1}\mathbf{y} \!\! \mod{p}$ are the generators of the cyclic factors, where $\mathbf{x} = \mathbf{r}_1\mathbf{r}_3\mathbf{r}_1\mathbf{r}_3\mathbf{r}_1\mathbf{r}_2$, $\mathbf{y} = \mathbf{r}_3\mathbf{r}_1\mathbf{r}_3\mathbf{r}_2\mathbf{r}_1\mathbf{r}_2$, while $\mathbf{r}_1 \!\! \mod{p}$ and $\mathbf{r}_3 \!\! \mod{p}$ are the generators of the dihedral factor. To show that $G_0^p$ corresponds to the first decomposition, we verify that the generators of the cyclic factors have the same order. 
    
    A quick calculation shows that for any positive integer $s$, we have
    \[
        \mathbf{x}^s =
        \begin{bmatrix}
            1 & 0 & 0 & 0 \\
            4s^2 & 1 + 2s & -4s & 0 \\
            2s^2 - 2s & s & 1 - 2s & 0 \\
            2s^2 & s & -2s & 1
        \end{bmatrix}.
    \]
    It follows that if $s$ is the order of $\mathbf{x} \!\! \mod{p}$, then $s$ must be the smallest non-negative odd rational prime that is divisible by $p$. That is, $s = 5, p, q$ according as $p$ is in Class I, II, III, respectively. The same conclusion applies to $\mathbf{x}^{-1}\mathbf{y} \!\! \mod{p}$ whose $s$th power is given by
    \[ 
        (\mathbf{x}^{-1}\mathbf{y})^s =
        \begin{bmatrix}
            1 & 0 & 0 & 0 \\
            4s^2 & 1 - 4s & 2s & 6s \\
            2s^2 + s & -2s & 1 + s & 3s \\
            2s^2 + s & -2s & s & 1 + 3s            
        \end{bmatrix}.  
    \]
    Thus, $\mathbf{x} \!\! \mod{p}$ and $\mathbf{x}^{-1}\mathbf{y} \!\! \mod{p}$ have the same order, completing the proof.
\end{proof}
From \cite{Hartley2006, McMullenSchulte2002}, we see that $\mathcal{A}_3^p$, $\mathcal{B}_3^p$, $\mathcal{H}_3^p$, and $[3, 6]_{(s, 0)}$ are all string C-groups. The first three correspond to the tetrahedron $\{3, 3\}$, the octahedron $\{3, 4\}$ or its dual cube $\{4, 3\}$, and the icosahedron $\{3, 5\}$ or its dual dodecahedron $\{5, 3\}$, respectively. The orders of these groups are listed in \textbf{Table \ref{tbl:relevantGroups}}.

Applying \textbf{Theorem \ref{thm:groupTypeGp023}} and employing an approach similar to what we demonstrated in \textbf{Example \ref{ex:536GroupType}} lets us classify the distinguished subgroup $G_2^p$. We summarize the results of the calculations together with the classification of $G_0^p$ and $G_3^p$ in \textbf{Table \ref{tbl:GpG023pClassificationOdd}}. Observe that the only time in which $G_2^p$ is not a reflection group for a non-singular space is when $k = 6$ and $p$ is an associate of $3$. 

\section{Star C-groups of type $\{5, 3; k\}$ with $k = 3, 4, 5, 6$}
    \label{sec:C-groups}
    
    We have already shown earlier that each of the reduced groups $G^p$ is a star C-group when $p$ is an even prime. We now prove that the same statement holds when $p$ is an odd prime. The key to proving this result, which we state more formally in a forthcoming theorem, relies on the following lemma:

    \begin{lemma}
        Let $p$ be an odd prime in $\mathbf{Z}[\tau]$ such that the group $G^p$, generated by the reductions modulo $p$ of the matrices $\mathbf{r}_0, \mathbf{r}_1, \mathbf{r}_2, \mathbf{r}_3$ in \eqref{eq:refMats}, is an orthogonal group of type $O(4, q, \varepsilon)$ or $O_1(4, q, \varepsilon)$. Then $G^p = \langle{\mathbf{r}_0, \mathbf{r}_1, \mathbf{z}, \mathbf{r}_3}\rangle^p$, where $\mathbf{z} = \mathbf{r}_3^{\mathbf{r}_2\mathbf{r}_1}, \mathbf{r}_0^{(\mathbf{r}_3\mathbf{r}_1\mathbf{r}_2)^5}, \mathbf{r}_3^{\mathbf{x}^i\mathbf{y}}$ according as $k = 4, 5, 6$, respectively. Here, $\mathbf{a}^{\mathbf{b}}$ denotes the conjugate $\mathbf{b}\mathbf{a}\mathbf{b}^{-1}$.
        \label{lem:53kPolys}
    \end{lemma}

    \begin{proof}
        To prove the lemma, it suffices to show that replacing $\mathbf{r}_2 \!\! \mod{p}$ by $\mathbf{z} \!\! \mod{p}$ still results in an orthogonal group of the same type as $G^p$. 

        \begin{figure}[H]
            \begin{subfigure}{\textwidth}
                \centering
                \includegraphics[scale = 0.40, keepaspectratio = true]{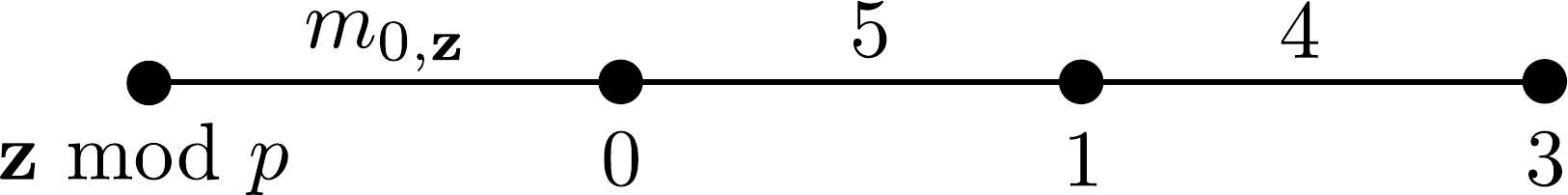}
                \vspace{2mm}
                \caption{$k = 4$, $\mathbf{z} = \mathbf{r}_3^{\mathbf{r}_2\mathbf{r}_1}$, $m_{0, \mathbf{z}} > 2$}
                \vspace{4mm}
                \label{fig:l54diagram}
            \end{subfigure}
            \begin{subfigure}{0.45\textwidth}
                \centering
                \includegraphics[scale = 0.40, keepaspectratio = true]{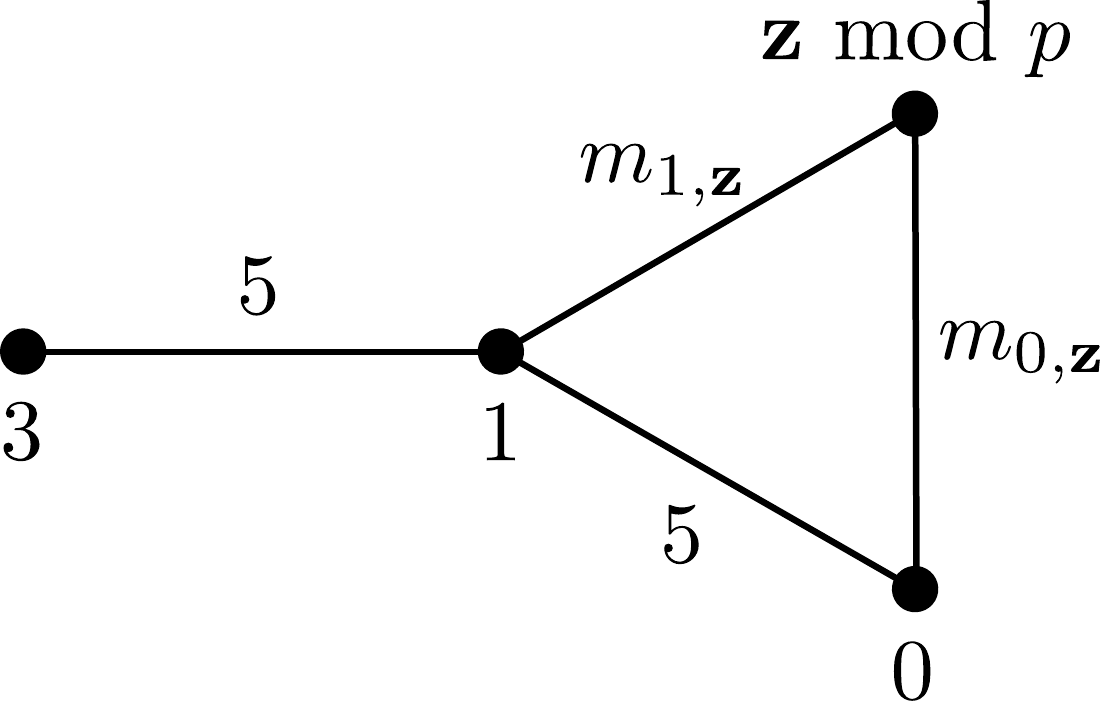}
                \vspace{2mm}
                \caption{$k = 5$, $\mathbf{z} = \mathbf{r}_0^{(\mathbf{r}_3\mathbf{r}_1\mathbf{r}_2)^5}$, $m_{1, \mathbf{z}} > 2$}
                \label{fig:lm55diagram}
            \end{subfigure}
            \begin{subfigure}{0.45\textwidth}
                \centering
                \includegraphics[scale = 0.40, keepaspectratio = true]{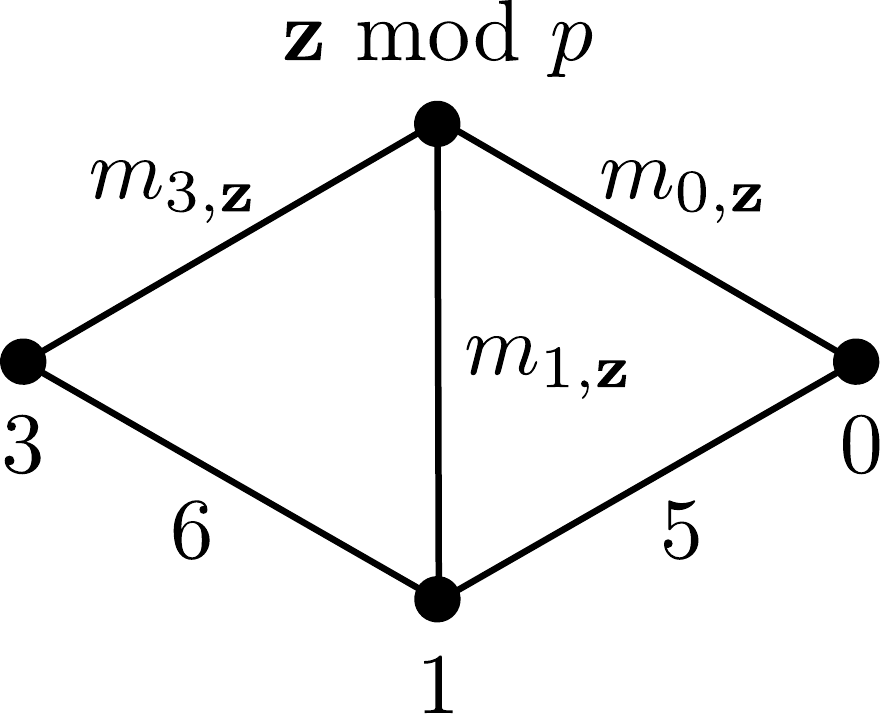}
                \vspace{2mm}
                \caption{$k = 6$, $\mathbf{z} = \mathbf{r}_3^{\mathbf{x}^i\mathbf{y}}$, $m_{1, \mathbf{z}} > 2$}
                \label{fig:lmt56diagram}
            \end{subfigure}   
            \caption{Coxeter diagram of $G^p_{\mathbf{r}_2 \leftarrow \mathbf{z}} = \langle{\mathbf{r}_0, \mathbf{r}_1, \mathbf{z}, \mathbf{r}_3 }\rangle^p$. The root of $\mathbf{z} \!\! \mod{p}$ is $2c_1 + 2c_2 + c_3$, $c_0 + (4 + 4\tau)c_1 + (2 + 2\tau)c_2 + (4 + 6\tau)c_3$, $6c_1 + 3c_2 + 4c_3$ according as $k = 4, 5, 6$, respectively.}
            \label{fig:lm5kdiagrams}
        \end{figure}

        In \textbf{Figure \ref{fig:lm5kdiagrams}}, we illustrate the diagram of the group $G^p_{\mathbf{r}_2 \leftarrow \mathbf{z}} := \langle{\mathbf{r}_0, \mathbf{r}_1, \mathbf{z}, \mathbf{r}_3}\rangle^p$ for the indicated value of $k$. Since the diagram is connected, we apply an analog of \textbf{Theorem \ref{thm:groupTypeGp}} to $G^p_{\mathbf{r}_2 \leftarrow \mathbf{z}}$, and conclude that $G^p_{\mathbf{r}_2 \leftarrow \mathbf{z}}$ is an orthogonal group of type either $O(4, q, \varepsilon)$ or $O_1(4, q, \varepsilon)$. A series of computations similar to what we did in \textbf{Example \ref{ex:536GroupType}} shows that $G^p$ and $G^p_{\mathbf{r}_2 \leftarrow \mathbf{z}}$ have the exact same orthogonal group type. Consequently, $G^p = G^p_{\mathbf{r}_2 \leftarrow \mathbf{z}} = \langle{\mathbf{r}_0, \mathbf{r}_1, \mathbf{z}, \mathbf{r}_3}\rangle^p$.
    \end{proof}

    The following theorem states the main result of this section.

    \begin{theorem}
        For any odd prime $p$ in $\mathbf{Z}[\tau]$, the group $G^p$, generated by the reductions modulo $p$ of the matrices $\mathbf{r}_0, \mathbf{r}_1, \mathbf{r}_2, \mathbf{r}_3$ in \eqref{eq:refMats}, is a C-group. 
        \label{thm:GpTailTriangleCGroupOdd}
    \end{theorem}

    \begin{proof}
        To prove the theorem, we make use of \textbf{Lemma \ref{lem:intConds4}} and show that $G_0^p$, $G_2^p$, and $G_3^p$ are string C-groups which satisfy the intersections $G_0^p \cap G_2^p = G_{0, 2}^p$, $G_0^p \cap G_3^p = G_{0, 3}^p$, and $G_2^p \cap G_3^p = G_{2, 3}^p$.

        By \textbf{Theorem \ref{thm:groupTypeGp023}} and the remark that follows, we conclude that $G_0^p$ and $G_3^p$ are string C-groups, regardless of $p$. We shall show that the same is true  for $G_2^p$. This is equivalent to showing that $G_{0, 2}^p \cap G_{2, 3}^p = \langle{\mathbf{r}_1}\rangle^p$ by \eqref{eq:intConds3}. Clearly, the order 2 subgroup $\langle{\mathbf{r}_1}\rangle^p$ is contained in the intersection $G_{0, 2}^p \cap G_{2, 3}^p$. Since $G_2^p$ is a smooth quotient of the Coxeter group $[5, k]$, where $k = 3, 4, 5, \text{ or } 6$, then $G_{0, 2}^p$ and $G_{2, 3}^p$ must be distinct dihedral groups of order $2k$ and $10$, respectively. By order consideration, their intersection must have order $2$. Hence, $G_{0, 2}^p \cap G_{2, 3}^p$ must be contained in $\langle{\mathbf{r}_1}\rangle^p$ as well.

        It then remains to show that the required intersections are satisfied. We shall discuss the proof that $G_0^p \cap G_2^p = G_{0, 2}^p$ completely. The proofs of the remaining two follow similar arguments and will be omitted for economy. The main strategy is to assume that $I_{0, 2}^p := G_0^p \cap G_2^p$ is a subgroup of $G_0^p$ that properly contains $G_{0, 2}^p$, then use the subgroup structure of $G_0^p$ to derive a contradiction. For $k = 4, 5, 6$, this erroneous assumption coupled with \textbf{Lemma \ref{lem:53kPolys}} will lead to the conclusion that $G^p = G_2^p$, thereby contradicting \textbf{Table \ref{tbl:GpG023pClassificationOdd}} which implies that $G_2^p$ must be a proper subgroup of $G^p$, and hence cannot contain $\mathbf{r}_2 \!\! \mod{p}$. This series of contradictions forces the conclusion that $I_{0, 2}^p = G_{0, 2}^p$. The details and contradiction derived for each value of $k$ are discussed below:
        \begin{enumerate}
            \sloppy
            \item[$k = 3$.] We have $G_0^p \simeq \mathcal{A}_3^p$ and $G_{0, 2}^p \simeq D_3$. Since $D_3$ is a maximal subgroup of $\mathcal{A}_3^p$, then $I_{0, 2}^p = G_0^p$. This clearly contradicts the fact that $\mathbf{r}_2 \!\! \mod{p}$, which is an element of $G_0^p$, is not in $I_{0, 2}^p$.
            
            \item[$k = 4$.] We have $G_0^p \simeq \mathcal{B}_3^p$ and $G_{0, 2}^p \simeq D_4$. The only subgroup of $G_0^p$ that properly contains $G_{0, 2}^p$ and does not contain $\mathbf{r}_2 \!\! \mod{p}$ is $\langle{\mathbf{r}_1, \mathbf{r}_3^{\mathbf{r}_2\mathbf{r}_1}, \mathbf{r}_3}\rangle^p$. If $I_{0, 2}^p = \langle{\mathbf{r}_1, \mathbf{r}_3^{\mathbf{r}_2\mathbf{r}_1}, \mathbf{r}_3}\rangle^p$, then $\mathbf{r}_3^{\mathbf{r}_2\mathbf{r}_1} \!\! \mod{p}$ must be in $G_2^p$, implying that $G_2^p = \langle{\mathbf{r}_0, \mathbf{r}_1, \mathbf{r}_3^{\mathbf{r}_2\mathbf{r}_1}, \mathbf{r}_3}\rangle^p$. An application of \textbf{Lemma \ref{lem:53kPolys}} yields the contradiction that $G^p = G_2^p$. 
            
            \item[$k = 5$.] The case where $p$ is a prime in Class I, and hence an associate of $\sqrt{5}$, can be easily shown to satisfy the required intersection condition (with the aid of GAP, for example). It remains to show that the same intersection condition is satisfied whenever $p$ is in Class II or III. We have $G_0^p \simeq \mathcal{H}_3^p$ and $G_{0, 2}^p \simeq D_5$. The only subgroup of $G_0^p$ that satisfies the requirements is $\langle{\mathbf{r}_1, (\mathbf{r}_3\mathbf{r}_1\mathbf{r}_2)^5, \mathbf{r}_3}\rangle^p$. If $I_{0, 2}^p = \langle{\mathbf{r}_1, (\mathbf{r}_3\mathbf{r}_1\mathbf{r}_2)^5, \mathbf{r}_3}\rangle^p$, then $(\mathbf{r}_3\mathbf{r}_1\mathbf{r}_2)^5 \!\! \mod{p}$ must be in $G_2^p$. It follows that $G_2^p$ contains the subgroup $\langle{\mathbf{r}_0, \mathbf{r}_1, \mathbf{r}_0^{(\mathbf{r}_3\mathbf{r}_1\mathbf{r}_2)^5}, \mathbf{r}_3 }\rangle^p$. A second application of \textbf{Lemma \ref{lem:53kPolys}} yields the same contradiction as in the previous case.

            \item[$k = 6$.] The case where $p$ is an associate of $3$ can be easily shown to satisfy the required intersection condition as well. We now assume that $p$ is not an associate of $3$. In this case, we have $G_0^p \simeq (C_s \times C_s) \rtimes D_6$ and $G_{0, 2}^p \simeq D_6$. Since $s$ is a rational prime, a subgroup $G_0^p$ that properly contains $G_{0, 2}^p$ and that does not contain $\mathbf{r}_2 \!\! \mod{p}$ must be of the form $\langle{\mathbf{w}, \mathbf{r}_1, \mathbf{r}_3}\rangle^p$, where $\mathbf{w}$ is either $\mathbf{x}$ or $\mathbf{x}^i\mathbf{y}$ for some $0 \leq i \leq s - 1$. If $I_{0, 2}^p = \langle{\mathbf{x}, \mathbf{r}_1, \mathbf{r}_3}\rangle^p$, then $\mathbf{r}_2 = \mathbf{r}_1\mathbf{r}_3\mathbf{r}_1\mathbf{r}_3\mathbf{r}_1\mathbf{x}$, implying that $\mathbf{r}_2 \!\! \mod{p}$ is in $I_{0, 2}^p$. If $I_{0, 2}^p = \langle{\mathbf{x}^i\mathbf{y}, \mathbf{r}_1, \mathbf{r}_3}\rangle^p$, on the other hand, then $G_2^p$ contains the element $\mathbf{r}_3^{\mathbf{x}^i\mathbf{y}} \!\! \mod{p}$, and hence contains the subgroup $\langle{\mathbf{r}_0, \mathbf{r}_1, \mathbf{r}_3^{\mathbf{x}^i\mathbf{y}}, \mathbf{r}_3}\rangle^p$. A final application of \textbf{Lemma \ref{lem:53kPolys}} yields the exact same contradiction as in the previous two cases.
        \end{enumerate}
    \end{proof}

    The strategy used in the proof above relied on the property that, for each value of $k$, at least two of the distinguished subgroups $G_0^p$, $G_2^p$, $G_3^p$ have relatively simple structure and possess subgroups that can be easily enumerated either manually or with the aid of GAP. In fact, when $p$ is an associate of the prime $\sqrt{5}$ or $3$, then the same proof strategy can be employed to show that the reduction modulo $p$ of the star Coxeter group $[5, 3; \infty]$ is a C-group. In this case, we may use the scaling factor $c_3 = 4\tau$, which corresponds to $\rho_{\infty} = 4$, to obtain the rescaled basis vector $v_3 = 4\tau a_3$. With this scaling, the order of $\mathbf{r}_1\mathbf{r}_3 \!\! \mod{p}$ becomes $5$ or $3$, and we obtain a star C-group of type $\{5, 3; 5\}$ or $\{5, 3; 3\}$, respectively.

\section{alternating semiregular 4-polytopes from $G^p$}
    \label{sec:semiregular}

    By applying the method of Wythoff construction developed in \cite{MonsonSchulte2012} to any star C-group $G^p$ in \textbf{Tables \ref{tbl:GpG023pClassificationEven}} and \textbf{\ref{tbl:GpG023pClassificationOdd}}, one may construct an alternating semiregular 4-polytope $\mathcal{S}$. Such a polytope possesses a 5-fold rotational symmetry and has possibly two distinct types of regular polyhedral cells $\mathcal{P}$ and $\mathcal{Q}$, with two of each alternating around an edge. The types of $\mathcal{P}$ and $\mathcal{Q}$ depend on which node in the diagram of $G$ is ringed. We briefly describe two examples below:
    \begin{example}
        Let us take the Coxeter group $G \simeq [5, 3; 3]$ and discuss the resulting alternating semiregular 4-polytope $\mathcal{S}$ for each of two specific prime moduli. 
        \begin{enumerate}
            \item For the even prime $p = 2$, we recall from \S{\ref{subsec:evenPrime}} that the reduced group $G^2$ is isomorphic to the group $C_2^4 \rtimes A_5$ of order 960. If we ring the node labeled 2 in the Coxeter diagram for $G$ (see \textbf{Figure \ref{fig:533diagramEven}}), then $\mathcal{S}$ will have 16 vertices (the right cosets of $G_2^2$), 120 edges (the right cosets of $G_1^2$), 160 triangular subfacets (the right cosets of $G_{0, 3}^2$), 16 hemi-icosahedral cells $\mathcal{P} = \{3, 5\}_5$ (the right cosets of $G_3^2$), and 40 tetrahedral cells $\mathcal{Q} = \{3, 3\}$ (the right cosets of $G_0^2$). 
            
            \begin{figure}[H]
                \begin{subfigure}{0.45\textwidth}
                    \centering
                    \includegraphics[scale = 0.40, keepaspectratio = true]{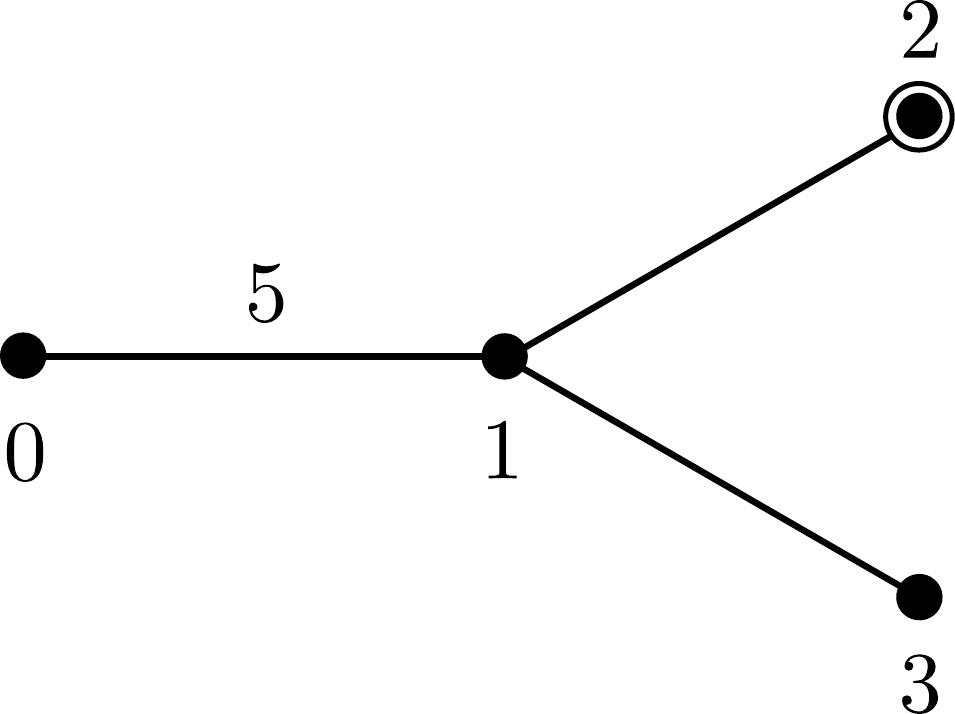}
                    \vspace{2mm}
                    \caption{}
                    \label{fig:533diagramEven}
                \end{subfigure}
                \begin{subfigure}{0.45\textwidth}
                    \centering
                    \includegraphics[scale = 0.40, keepaspectratio = true]{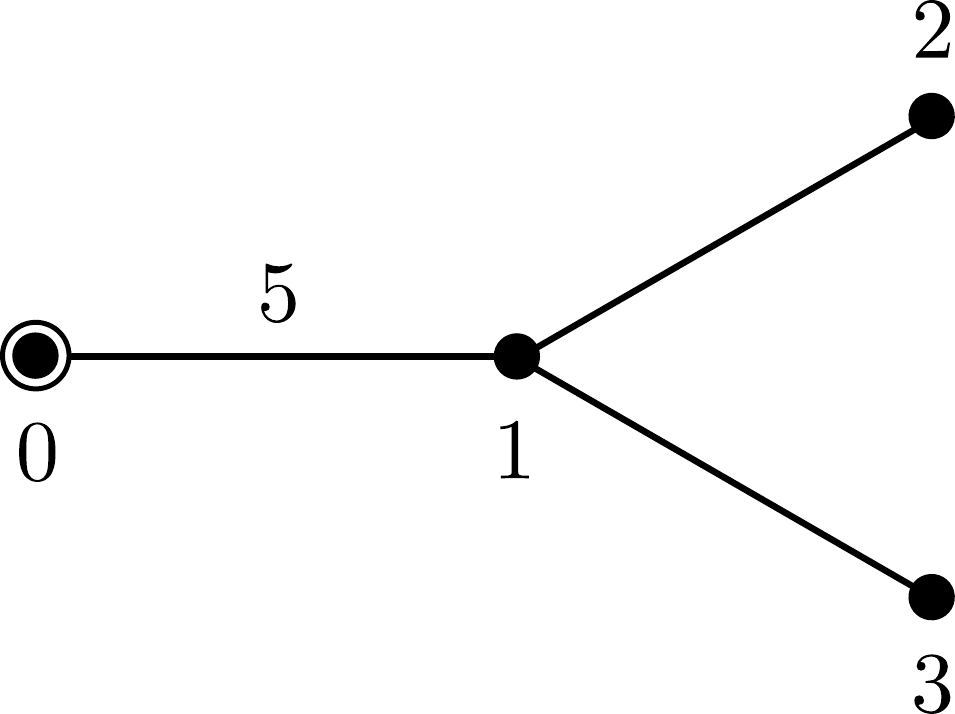}
                    \vspace{2mm}
                    \caption{}
                    \label{fig:533diagramOdd}
                \end{subfigure}   
                \caption{Coxeter diagram of the group $G \simeq [5, 3; 3]$ with a ringed node.}
                \label{fig:533diagram}
            \end{figure}

            Each vertex of $\mathcal{S}$ is surrounded by six hemi-icosahedra and ten tetrahedra. Moreover, each edge of $\mathcal{S}$ is incident to a pair of hemi-icosahedral cells and a pair of tetrahedral cells which are arranged alternately around it. Finally, since $\mathcal{S}$ has two types of cells, \textbf{Proposition 4.8(b)} of \cite{MonsonSchulte2012} implies that $\mathcal{S}$ is a 2-orbit non-regular 4-polytope whose automorphism group is $G^2$.  
            
            \item For the odd prime $p = \sqrt{5}$, we see in \textbf{Table \ref{tbl:GpG023pClassificationOdd}} that the reduced group $G^{\sqrt{5}}$ is the orthogonal group $O_1(4, 5, -1)$ of order 15,600. If we ring the node labeled 0 this time (see \textbf{Figure \ref{fig:533diagramOdd}}), then $\mathcal{S}$ will have 650 vertices (the right cosets of $G_0^{\sqrt{5}}$), 1,950 edges (the right cosets of $G_1^{\sqrt{5}}$), 1,560 pentagonal subfacets (the right cosets of $G_{2, 3}^{\sqrt{5}}$), and 260 dodecahedral cells $\mathcal{P} = \mathcal{Q} = \{5, 3\}$ (the right cosets of $G_3^{\sqrt{5}}$ and $G_2^{\sqrt{5}}$). Four such cells surround an edge of $\mathcal{S}$. 

            Observe that, unlike in the previous case, $\mathcal{S}$ has only one type of cell this time. Indeed, since the map which swaps $\mathbf{r}_2 \!\! \mod{p}$ and $\mathbf{r}_3 \!\! \mod{p}$, and fixes $\mathbf{r}_0 \!\! \mod{p}$ and $\mathbf{r}_1 \!\! \mod{p}$ induces an automorphism of $G^{\sqrt{5}}$, \textbf{Proposition 4.8(a)} of \cite{MonsonSchulte2012} implies that $\mathcal{S}$ is a regular 4-polytope of type $\{5, 3, 4\}$ whose vertex-figure is the octahedron $\{3, 4\}$. Furthermore, its automorphism group is isomorphic to the much larger group $G^{\sqrt{5}} \rtimes C_2$.
        \end{enumerate}
        \label{ex:533SemiregularPolytope}
    \end{example}

\bibliographystyle{amsplain}

\end{document}